\documentclass{article}
\usepackage[utf8]{inputenc}
\usepackage{amsthm,amssymb}
\usepackage{amsmath}
\usepackage{latexsym}
\usepackage{url}
\usepackage[small,bf,hang]{caption} 
\usepackage{algorithm}
\usepackage{algorithmic}
\usepackage[pdftex]{graphicx}

\usepackage{etoolbox}
\makeatletter
\patchcmd{\@sect}{#8}{\boldmath #8}{}{}\patchcmd{\@sect}{#8}{\boldmath #8}{}{}

\patchcmd{\@sect}{#8}{\boldmath #8}{}{}

\let\ori@chapter\@chapter
\def\@chapter[#1]#2{\ori@chapter[\boldmath#1]{\boldmath#2}}
\makeatother

\newtheorem{problem}{Problem}
\newtheorem{theorem}{Theorem}
\newtheorem{lemma}[theorem]{Lemma}

\title{\bf On Some Generalized\\
Vertex Folkman Numbers}

\author{Zohair Raza Hassan\\[-0.2ex]
\small Department of Computer Science, {\tt zh5337@rit.edu}\\[-0.6ex]
\small Rochester Institute of Technology, Rochester, NY 14623, USA\\[1.0ex]\and
Yu Jiang\\[-0.3ex]
\small College of Electronics and Information Engineering, {\tt wdjiangyu@126.com}\\[-0.6ex]
\small Beibu Gulf University, Qinzhou 535011, P.R. China\\[1.0ex]\and
David E. Narváez\\[-0.3ex]
\small Computer Science Department, {\tt david.narvaez@rochester.edu}\\[-0.6ex]
\small University of Rochester, Rochester, NY 14627, USA\\[1.0ex]\and
Stanis{\l}aw Radziszowski\\[-0.3ex]
\small Department of Computer Science, {\tt spr@cs.rit.edu}\\[-0.6ex]
\small Rochester Institute of Technology, Rochester, NY 14623, USA\\[1.0ex]\and
Xiaodong Xu\\[-0.3ex]
\small Guangxi Academy of Sciences, {\tt xxdmaths@sina.com}\\[-0.6ex]
\small Nanning 530007, P.R. China\\[1ex]
}

\date{November 14, 2022}
\begin{document}
\maketitle
\thispagestyle{empty}

\begin{abstract}
For a graph $G$ and integers $a_i\ge 1$,
the expression $G \rightarrow (a_1,\dots,a_r)^v$ means that
for any $r$-coloring of the vertices of $G$ there exists a monochromatic
$a_i$-clique in $G$ for some color $i \in \{1,\cdots,r\}$.
The vertex Folkman numbers are defined as
$F_v(a_1,\dots,a_r;H) = \min\{|V(G)| : G$ is $H$-free and
$G \rightarrow (a_1,\dots,a_r)^v\}$, where $H$ is a graph.
Such vertex Folkman numbers have been extensively studied
for $H=K_s$ with $s>\max\{a_i\}_{1\le i \le r}$.
If $a_i=a$ for all $i$, then we use notation
$F_v(a^r;H)=F_v(a_1,\dots,a_r;H)$.

Let $J_k$ be the complete graph $K_k$ missing one edge,
i.e. $J_k=K_k-e$.
In this work we focus on vertex Folkman numbers with $H=J_k$,
in particular for $k=4$ and $a_i\le 3$.
A result by Ne\v{s}et\v{r}il and R\"{o}dl from 1976
implies that $F_v(3^r;J_4)$ is well defined for any $r\ge 2$.
We present a new and more direct proof of this fact.
The simplest but already intriguing case is
that of $F_v(3,3;J_4)$, for which we establish the upper
bound of 135 by using the $J_4$-free process. We obtain the exact values and bounds
for a few other small cases of $F_v(a_1,\dots,a_r;J_4)$ when
$a_i \le 3$ for all $1 \le i \le r$, including
$F_v(2,3;J_4)=14$,
$F_v(2^4;J_4)=15$, and
$22 \le F_v(2^5;J_4) \le 25$.
Note that $F_v(2^r;J_4)$ is the smallest number of vertices
in any $J_4$-free graph with chromatic number $r+1$.
Most of the results were obtained with the help of
computations, but some of the upper bound graphs we found
are interesting by themselves.
\end{abstract}
\noindent
{\bf Keywords:}
Folkman number, vertex coloring\\
{\bf AMS classification subjects:} 05C55

\section{Introduction}

\subsection{Notation and Background}

All graphs considered in this paper are finite undirected
simple graphs. The order of the largest independent set in
graph $G$ will be denoted by $\alpha(G)$,
and the chromatic number of $G$ by $\chi(G)$.
Let $J_k$ be the complete graph $K_k$ missing one edge,
i.e., $J_k=K_k-e$. Note that $J_3$ is the path on 3 vertices,
and the diamond graph $J_4$ is formed by two triangles
sharing an edge.

For a graph $G$ and integers $a_1,\dots,a_r$,
such that $a_i\ge 1$ for $1 \le i \le r$,
the expression $G \rightarrow (a_1,\dots,a_r)^v$
means that for any $r$-coloring of the vertices of
$G$ there exists a monochromatic $a_i$-clique in $G$
for some color $i \in \{1,\cdots,r\}$.
In this paper, we call this property vertex-arrowing,
or simply arrowing.
It should be noted that an analogous
edge-arrowing property $G \rightarrow (a_1,\dots,a_r)^e$
is the basis of widely studied Ramsey edge-arrowing
problems. In particular, the Ramsey number
$R(a_1,\dots,a_r)$ is defined as the smallest integer $n$
such that $K_n \rightarrow (a_1,\dots,a_r)^e$.
All our results address vertex-arrowing,
but in some places we will refer to edge-arrowing
for comparison, context, or when used as a tool.
The vertex Folkman numbers $F_v$ are defined as
$$F_v(a_1,\dots,a_r;H) = \min\{|V(G)| : G \mathrm{\ is\ }
H\mathrm{-free\ and\ }
G \rightarrow (a_1,\dots,a_r)^v\},$$
where $H$ is a graph and $G$ is called $H$-free if it does not contain $H$ as a subgraph (in contrast to a commonly used definition of $H$-free where it is required that $G$ does not contain $H$ as an induced subgraph).
If $a_i=a$ for all $i$, then we use a more compact
notation, $F_v(a^r;H)=F_v(a_1,\dots,a_r;H)$.
The set of all $H$-free
graphs satisfying the arrowing
$G \rightarrow (a_1,\dots,a_r)^v$
will be denoted by
$\mathcal{F}_v(a_1,\dots,a_r;H)$ and we will call it
the set of Folkman
graphs for the corresponding parameters. Further,
$\mathcal{F}_v(a_1,\dots,a_r;H;n)$ will denote a subset of
the latter when restricted to graphs on $n$ vertices. 

Let us set $m=1+\sum_{i=1}^{r}(a_i-1)$. We will often use
the following lemma by Nenov \cite{Nen00} stating a simple
necessary condition for vertex arrowing to hold.
\begin{lemma}(Nenov 1980 \cite{Nen00})
\label{l:chromatic}
If $G \rightarrow (a_1,\dots,a_r)^v$, then $\chi(G) \ge m$.
\end{lemma}

Observe that if $a_i=2$ for all $i$, $1 \le i \le r$,
then the converse is also true.
In the
terms of Folkman numbers, this means that $F_v(2^r;H)$
is equal to the smallest number of vertices in any
$H$-free graph $G$ with $\chi(G)=r+1=m$. Note also that if $a_i=1$ for any color $i$, then essentially this color can be disregarded.

The vertex Folkman numbers have been extensively studied
when the avoided graph is complete, i.e., when $H=K_s$
(see \cite{Folk70,Nen84,LRU01,DR11,Fv445,Nen03,Nen09}).
They are well defined
when $s>\max\{a_i\ |\ 1\le i \le r\}$, since it is known
that for such $s$ the minimum in the definition
ranges over a nonempty set of graphs.
The situation is easy for $s \ge m$. Moreover, much is known
about vertex Folkman numbers and the corresponding Folkman graphs
when $s$ is close to, but less than, $m$. However, even
some of the basic questions become difficult for small $s$,
such as $s=3$ or $s=4$. One of the famous problems
which can be stated in these terms is the task of
finding the smallest
triangle-free graph with given chromatic number $r$, which
is equal to $F_v(2^{r-1};K_3)$. See the following subsection
for references and more details about this problem
in relation to our current work.
A recent Ph.D. thesis by Bikov \cite{Bik18} presents
a variety of Folkman problems, focusing on a computational
approach together with the known values and bounds for
Folkman numbers.

For graphs $F$ and $G$, the Ramsey number $R(F,G)$ is the smallest
integer $n$ such that if the edges of $K_n$ are 2-colored,
say red and blue, then necessarily this coloring includes
a copy of red-colored $F$ or a copy of blue-colored $G$.
If $F$ and $G$ are complete graphs, then we write
$R(s,t)$ instead of $R(K_s,K_t)$. A regularly updated survey
{\em Small Ramsey Numbers} \cite{SRN} contains the known bounds
and values for a variety of Ramsey numbers.

In this work we focus on the vertex Folkman numbers for
graphs avoiding $H=J_k$, in particular for $k=4$ and
$2 \le a_i\le 3$. Note that this special case of $J_4$-free
graphs admits some triangles, but not too many, since
in $J_4$-free graphs each edge can belong to at most
one triangle. We observe that avoiding $J_4$ falls
in-between the two extensively studied classical cases of
avoiding $K_3$ and $K_4$.
Also for $H=J_4$, we see
that $F_v(2^r;J_4)$ is the smallest number of vertices
in any $J_4$-free graph with chromatic number $r+1$.
These numbers are clearly well defined since
any $K_3$-free graph is also $J_4$-free, and
$F_v(2^r;K_3)$ is well defined for every $r \ge 1$.
Furthermore, the classical results for multicolor
Folkman numbers by Ne\v{s}et\v{r}il and R\"{o}dl
\cite{NR76,NR81} guarantee the existence of
$F_v(3^r;K_s)$ for $s \ge 4$ and of $F_v(3,3;J_4)$,
while little less known results by
Ne\v{s}et\v{r}il and R\"{o}dl \cite{NR76a}
and R\"{o}dl and Sauer \cite{RS92} imply also
the existence of $F_v(3^r;J_4)$ with $r>2$.

\subsection{Summary of Contributions}
\label{s:background}

This subsection summarizes our contributions.
They consist mainly of new lower and upper bounds
for some concrete Folkman numbers, which were obtained
with the help of computations (Theorems 3--7),
but also new interesting graph constructions. Several special
graphs were found during our computations, 
but we think that some of them (see Figures 1--4) are interesting just by themselves.

We start with Theorem \ref{t:3rj4}, which is theoretical.
This result is not new, as it is implied
by more general old results by Ne\v{s}et\v{r}il and R\"{o}dl
\cite{NR76,NR76a,NR81}, and more directly it follows
from Theorem 1.2 in the work
by R\"{o}dl and Sauer \cite{RS92}. In the next section,
we include our proof of Theorem \ref{t:3rj4}
addressing directly the case of avoiding $J_4$.

\begin{theorem}
\label{t:3rj4}
$F_v(3^r;J_4)$ is well defined for all $r \ge 1$.
\end{theorem}

The $J_4$-free graphs satisfying the required arrowing
property in Theorem~\ref{t:3rj4} quickly become very large
as $r$ grows.
The simplest, but already intriguing case, is
that of $F_v(3,3;J_4)$, for which we establish the upper
bound of 135 in Theorem~\ref{t:33j4}. Note that, by monotonicity,
Theorem~\ref{t:3rj4} implies the existence of Folkman numbers
of the form $F_v(a_1,\dots,a_r;J_4)$ with
$a_i \le 3$ for all $1 \le i \le r$.

A $J_4$-free graph $G$ is called {\em maximal}
$J_4$-{\em free}
if the addition of any edge creates a $J_4$ in $G$. A graph
$G$ for which $G \rightarrow (a_1,\dots,a_r)$
is called {\em minimal} if after the deletion of any edge
this arrowing does not hold. If $G$ is maximal and minimal, it is referred to as {\em bicritical}.
Using computational methods,
we obtain the exact values and bounds
for several small cases,
as stated in Theorems~\ref{t:222j4}--\ref{t:33j4} below.

\begin{theorem}
\label{t:222j4}
$F_v(2^3;J_4)=9$, and there are exactly $3$ graphs in
$\mathcal{F}_v(2^3;J_4;9)$, of which $1$ is maximal
and $1$ is minimal.
\end{theorem}

\begin{theorem}
\label{t:2222j4}
$F_v(2^4;J_4)=15$, and there are exactly $5$ graphs in
$\mathcal{F}_v(2^4;J_4;15)$, of which $1$ is maximal
and $2$ are minimal.
\end{theorem}

\begin{theorem}
\label{t:22222j4}
$22 \le F_v(2^5;J_4) \le 25$.
\end{theorem}

\begin{theorem}
\label{t:23j4}
$F_v(2,3;J_4)=14$ and there are exactly $212$ graphs in
$\mathcal{F}_v(2,3;J_4;14)$, of which $24$ are maximal,
$26$ are minimal, and $1$ is bicritical.
\end{theorem}

\begin{theorem}
\label{t:33j4}
$F_v(3,3;J_4) \le 135$.
\end{theorem}

For the context and comparison with the cases involving
$K_3$ and $K_4$ instead of $J_4$, we collect the values
and bounds from Theorems~\ref{t:222j4}--\ref{t:22222j4} in Table~\ref{tab:k3k4}.
Observe that since $K_3 \subset J_4 \subset K_4$, we
must have
$F_v(2^r;K_3) \ge F_v(2^r;J_4) \ge F_v(2^r;K_4)$,
for each $r$.

\begin{table}[h]
  \centering
  { \renewcommand{\arraystretch}{1.05} \small
    \begin{tabular}{|c||cc|c|cc|}
	\hline
    $r$ & $K_3$ & ref. & $J_4$ & $K_4$ & ref. \\
	\hline
2& 5 & $C_5$& 3 & 3 & $K_3$\\
3& 11 & \cite{Chv74} & {\bf 9} & 6 & $W_6$\\
4& 22 & \cite{JG95} & {\bf 15} & 11 & \cite{Nen84}\\
5& 32--40 & \cite{Goe20} & {\bf 22 -- 25} & 16 & \cite{LR11}\cite{Nen07}\\
	  \hline
    \end{tabular}
	\caption{Known values and bounds for $F_v(2^r;H)$,
	for $r\le 5$ and $H \in \{K_3,J_4,K_4\}$.
	The bold entries in the $J_4$ column were obtained in this work.
	For easy entries we give the upper bound witness graph.
	The unique witness for $F_v(2^3;K_4)=6$ is the wheel graph
	$W_6=K_1+C_5=K_6-C_5$.}
	\label{tab:k3k4}
    }
\end{table}

The following sections contain our proof of Theorem~\ref{t:3rj4},
the description of computations leading to
Theorems~\ref{t:222j4}--\ref{t:33j4},
and graph constructions.
The closing section states some open problems and
it contains a few remarks for parameters beyond
those studied in this paper.
The witness graphs for Theorems~\ref{t:222j4}--\ref{t:33j4},
as well as the code implementing algorithm $\textbf{A}$ in
Section~\ref{s:algorithms} are available at
\url{https://www.cs.rochester.edu/~dnarvaez/folkmanj4/}.

We found a few discrepancies between our results
and those claimed in the paper~\cite{JLX19}. We investigated
all such differences, and we arrived to the conclusion
that the computations reported in~\cite{JLX19} were
incomplete. 
The computational results reported in this paper were obtained
by independent implementations which agreed on the final and
intermediate claims. The main implementation described throughout
this paper produced all reported results, except those involving
SAT-solvers in Section 3.4. Two others implementations were used
to corroborate various parts of these results and to manage the
interface in Section 3.4.

\section{The Existence of $F_v(3^r;J_4)$ and $F_e(3,3;H)$}

Two seminal papers by Ne\v{s}et\v{r}il and R\"{o}dl~\cite{NR76,NR81}
lay the foundation for our reasoning
in this section: the first one from 1976 implies that
the edge Folkman numbers $F_e(3^r;K_4)$ are well defined
for all $r \ge 1$, and the second paper from 1981 shows
that $F_v(3,3;J_4)$ is well defined. These, together with
a technique developed by Dudek and R\"{o}dl in 2008~\cite{DR08},
permit us to give a rather elementary proof that
$F_v(3^r;J_4)$ is well defined for all $r \ge 1$.

For graph $G = (V_G,E_G)$, we define the graph $F=DR(G)$
as in the construction by Dudek-R\"{o}dl \cite{DR08}, as follows:
$$DR(G) = F = (E_G,E_F),$$
where the set of vertices of $F$ consists of
the edges of $G$, and the edge set $E_F$ contains
the edges $\{ef, fg, eg\}$
for every edge-triangle $efg$ in $E_G$, and no other edges.
Note that each pair of edges from triangle
$\{e, f, g\} \subseteq E_G$
spans the same three vertices in $G$,
and thus the same three vertices in
the corresponding vertex-triangle in $F$.
Such triangles in $F$ will be called {\em images}
of triangles from $G$,
other triangles in $F$ will be called {\em spurious}.
Note that for any edge $ef$ in $E_F$, the edges $e$ and $f$ in $E_G$
must share one vertex.
It is easy to observe (see the proof of Lemma~\ref{l:DRG}(2) below)
that two image triangles may share vertices but
no edges.

\smallskip
\noindent
\emph{Example.}
For $G=K_4$, $DR(G)$ has 6 vertices, 12 edges
(it is 4-regular), and 8 triangles. These 8 triangles
are split into 4 images of triangles from $G$
and 4 spurious triangles.

\begin{lemma}
\label{l:DRG}
Let $G$ be any $K_4$-free graph, and let $F$ denote
the graph $DR(G)$. Then we have that:

\begin{enumerate}
    \item $F$ has no spurious triangles,
    \item $F$ is $J_4$-free, and
    \item $G \rightarrow (3^r)^e$ if and only if $F \rightarrow (3^r)^v$,
for every $r \ge 1$.
\end{enumerate}
\end{lemma}

\begin{proof}
First, we will show that the graph $F=DR(G)$
has no spurious triangles.
For contradiction, suppose that $efg$ is a spurious triangle in $F$,
where $e=\{A,B\}$ and $f=\{A,C\}$ for some vertices $A,B,C \in V_G$.
Since edge $g$ is incident to both $e$ and $f$, but $efg$ is not
an image of a triangle in $G$, we must have $g=\{A,D\}$ for
another vertex $D \in V_G$. This implies that $ABC, ABD, ACD$ are
triangles, and thus also $BCD$, in $G$. Hence $ABCD$ forms a $K_4$
in $G$, contradicting the assumption. This shows part (1).

If $F$ contains $J_4$ with the vertices $\{e,f,g,h\}$ and
formed by two triangles $\{efg\}$ and $\{efh\}$, then we
claim that at most one of them is an image triangle.
In order to see this, let $e=\{A,B\}$, $f=\{A,C\}$ and note
that the unique image triangle in $F$ containing $e$ and $f$
is the one implied by the triangle $ABC$ in $G$.
Hence, at least one of $\{efg\}$ and $\{efh\}$ must
be spurious, but by (1) this is impossible in $F=DR(G)$
obtained from a $K_4$-free graph $G$. Thus (2) follows.

For (3),
consider the natural bijection between all
$r$-vertex-colorings of $F$ and $r$-edge-colorings
of $G$. This bijection
preserves the number of colors used in any edge-triangle
in $G$ when mapped to its image triangle
in $F$. Hence, because of (1) and (2), we can conclude (3).
\end{proof}

Using Lemma \ref{l:DRG}, we can give a simple proof
that for every $r$ there exists a $J_4$-free graph
such that in any $r$-coloring of its vertices
there must be a monochromatic triangle, or
equivalently, $\mathcal{F}_v(3^r;J_4) \not=\emptyset$.

\noindent
{\em Proof of Theorem~\ref{t:3rj4}}.
A general result by Ne\v{s}et\v{r}il and R\"{o}dl~\cite{NR76}
implies that the sets $\mathcal{F}_e(3^r;K_4)$ are
nonempty for all $r \ge 1$. This, together with the claim
that $\mathcal{F}_v(3^2;J_4) \not=\emptyset$, was also
discussed in~\cite{XLR18}. For general $r$, consider
any graph $G \in \mathcal{F}_e(3^r;K_4)$.
Then by Lemma~\ref{l:DRG}(3), we have that
$DR(G) \in \mathcal{F}_v(3^r;J_4)$,
and thus the numbers of the form
$F_v(3^r;J_4)$ are well defined for all $r \ge 1$.
\hfill$\square$

Note that, by monotonicity, Theorem~\ref{t:3rj4} implies
the existence of Folkman numbers of the form
$F_v(a_1,\dots,a_r;J_4)$ with
$1 \le a_i \le 3$ for all $1 \le i \le r$.
The orders of graphs in $\mathcal{F}_e(3^r;K_4)$
and $\mathcal{F}_v(3^r;J_4)$ can be expected to be quite
large, even for small $r$. In the trivial case for $r=1$
we have $F_e(3;K_4)=F_v(3;J_4)=3$, but both problems become
very difficult already for $r=2$. For the edge problem,
the best known bounds are $21 \le F_e(3,3;K_4) \le 786$~\cite{BN20,LRX14},
while for the vertex problem we establish
the bound $F_v(3,3;J_4) \le 135$ in Theorem~\ref{t:33j4}.
We are not aware of any reasonable bounds
for $r \ge 3$ in either case.

We wish to point to a study of the existence
of edge Folkman numbers for some small parameters \cite{XLR18}. While a simple argument
easily shows that $F_e(3,3;J_4)$ does not exist,
for other cases with $|V(H)|\le 5$ one can prove
or disprove the existence of $F_e(3,3;H)$
with some work, leaving only two open cases.
Namely, the following is known:
the sets $\mathcal{F}_e(3,3;H)$ are nonempty for
all connected graphs $H$ containing $K_4$, and for
some graphs not containing $K_4$.
If $H$ is any connected $K_4$-free graph on
$5$-vertices containing $K_3$,
then $\mathcal{F}_e(3,3;H)=\emptyset$
except for two possible cases:
the wheel graph $W_5=K_1+C_4$ and its subgraph
$\overline{P_2 \cup P_3} \subset W_5$.
The latter two cases remain open.

\section{Computational Proofs}

\subsection{Overview and Algorithms}
\label{s:algorithms}

In this section we describe the computations which were performed
to obtain the proofs of Theorems~\ref{t:222j4}--\ref{t:33j4} stated in Section~\ref{s:background}.
First, we give an overview of the algorithms that were
used or developed for this work,
including some details of their implementation.
In the following subsections
we summarize the results of our computations.
We present graphs establishing the upper bounds in Theorems 3--6,
and give counts for several intermediate
graph families which were obtained.
All graphs involved in the computations were $J_4$-free.
The target sets of graphs had additional constraints consisting
of the number of vertices, independence number, chromatic number,
and the desired parameters of arrowing, $\{a_1,\dots,a_r\}$,
where $2 \le a_i \le 3$ for $1 \le i \le r$.

The basis of our software framework consisted of the package
{\tt nauty} developed by McKay \cite{nauty}, which includes
a powerful graph generator {\tt geng}, tools to remove graph
isomorphs, and several other utilities for graph manipulation.
In the following, we will list some of the graphs in their
g6-format, a compact string representation of graphs in {\tt nauty}.
These graphs are also available at \url{https://www.cs.rochester.edu/~dnarvaez/folkmanj4/}.

The template of our main {\bf Extension
Algorithm  A} is presented 
and commented on below. We also implemented
filters for extracting graphs with specified chromatic number,
graph which are maximal $J_4$-free, those which arrow
$(2,3)^v$ and $(3,3)^v$, and other utilities.
Observe that by Lemma~\ref{l:chromatic} the test for arrowing $(2^r)^v$ is
the same as for chromatic number.

The graph families pointed to in Table~\ref{t:smallcases} were obtained
by using {\tt geng} with filters for $J_4$-free graphs and
for graphs with given chromatic number. For graph families on
13 or more vertices, we used mainly algorithm {\bf A}
together with other utilities,
as described in the notes to Table~\ref{t:nonisomax}.

Our custom filter for graphs with specified chromatic number range
was tuned to process large number of graphs with small $\chi(G)$.
For a given graph $G$, first we find all maximal independent sets,
and then determine $\chi(G)$ as the minimum number of these independent
sets which cover $V_G$.
The custom filter for maximal $J_4$-free graphs has two modes:
a full test detecting graphs for which addition of any edge forms a $J_4$,
and a partial test for graphs being constructed within algorithm {\bf A}
which cannot be maximal $J_4$-free after {\bf A} terminates.
The latter permitted to significantly prune the output of {\bf A},
which was then filtered through the full test.

Testing whether $G \rightarrow (2,3)^v$ was applied only to
graphs with $\chi(G) \ge 4$, since by Lemma 1 this arrowing
does not hold if $\chi(G) \le 3$. This test was done by
checking that for every maximal independent set $I \subset V_G$
the set of vertices $(V_G \setminus I)$ does not
induce any triangle.
The test for $G \rightarrow (2,2,3)^v$ was accomplished
similarly by checking that the set of vertices
$(V_G \setminus (I_1 \cup I_2))$ does not induce any triangle,
for every pair of maximal independent sets $I_1$ and $I_2$.
The test for $G \rightarrow (3,3)^v$ was accomplished
with a totally distinct approach involving SAT-solvers
as described in Section 3.4. This was necessary since
for arrowing $(3,3)^v$ we were processing a large number of graphs
on 100 to 200 vertices, for which an effective handling
of their chromatic number and enumeration of maximal
triangle-free sets would
be computationally very expensive.

Next, we give a high-level pseudocode of the {\bf Extension Algorithm A}
followed by comments explaining some of its components
in more detail.

\renewcommand{\algorithmicrequire}{\textbf{Input:}}
\renewcommand{\algorithmicensure}{\textbf{Output:}}

\begin{algorithm}[H]
\caption*{{\bf Extension Algorithm A}}
  \begin{algorithmic}
\REQUIRE{
$\mathcal{G}$ - a set of $J_4$-free graphs, each on $n$-vertices,
$q$ - extension degree,\\
$\chi$ - target chromatic number,
$\delta$ - minimum cone size.}
\ENSURE{$\mathcal{H}$ - a set graphs which are extensions of 
graphs from $\mathcal{G}$.
$H \in \mathcal{H}$ if and only if $H$ is a
$q$-vertex extension of any graph
$G \in \mathcal{G}$, $q$ new vertices in $H$ form an independent set,
$|V_H|=n+q$, new vertices have degree $\ge\delta$,
and such that $H$ is maximal $J_4$-free
and $\chi(H)\ge\chi$.}

\STATE $\mathcal{H}=\emptyset$;
\FOR{every graph $G \in \mathcal{G}$}
\STATE (for definitions of feasible cones and $\tau$ see {\bf Comments}
part in the main text below)
\STATE Compute and store $\mathcal{C}_G$, the set of feasible cones in $G$
of size at least $\delta$;
\STATE Compute and store the values of $\tau(C,D)$, for all $C,D \in \mathcal{C}_G$;

If $q=1$ then form $\mathcal{H}$ from $G$ and $\mathcal{C}_G$;

If $q=2$ then form $\mathcal{H}$ from $G$ and pairs of cones in $\mathcal{C}_G$ using $\tau$;

If $q \ge 3$ then
\FOR{$k=3$ to $q$}
\STATE Using known $(k-1)$-tuples, make all $k$-tuples
($k$-multisets) of feasible cones
$\{C_1,\dots,C_k\}$ such that $\tau(C_i,C_j)$
is true for all $1 \le i, j \le q$.
If $k=q$, then for each such $q$-tuple
make graph $H$ from $G$ and $\{C_1,\dots,C_q\}$.\\
If $H$ is maximal $J_4$-free, then add $H$ to $\mathcal{H}$.
\ENDFOR
\ENDFOR

\STATE If $q<3$ then remove from $\mathcal{H}$ graphs $H$ which are not maximal $J_4$-free;
\STATE Remove isomorphs from $\mathcal{H}$;
\STATE Remove from $\mathcal{H}$ graphs $H$ with $\chi(H)<\chi$.
  \end{algorithmic}
\end{algorithm}

\noindent
{\bf Comments on the Extension Algorithm A}.
The inputs are: a family
of graphs $\mathcal{G}$ consisting of $n$-vertex $J_4$-free graphs,
an integer $q$, which is the extension degree, the target chromatic
number $\chi$, and the minimum degree $\delta$ of new vertices.
For each $G \in \mathcal{G}$, algorithm $\bf A$ outputs 
all maximal $J_4$-free graphs $H$ with $\chi(H) \ge \chi$
such that they can be obtained from $G$ by adding an independent set
$I=\{v_1,\dots,v_q\}$ and some edges between $I$ and $V_G$.
New vertices have degree at least $\delta$.
These output graphs $H$ will be called $q$-vertex extensions of
the input graph $G$.

The new edges of $H$
are defined by $q$ {\em cones} $\{C_1,\dots,C_q\}$, $C_i \subseteq V_G$,
where the set of edges connecting $v_i$ to $V_G$ is
$\{\{v_i,u\}\ |\ u \in C_i\}$. First, we precompute the set of
all {\em feasible cones} $\mathcal{C}_G$ such that for each $C\in\mathcal{C}_G$
the 1-vertex extension of $G$ using $C$ is $J_4$-free and $|C|\ge \delta$.
We also precompute a binary predicate $\tau(C,D)$ on pairs of feasible
cones which is false if
$C\cap D=\emptyset$, $C\cap D=\{x\}$ and there is
no vertex $y \in (C\setminus D)\cup(D\setminus C)$ connected to $x$,
or $C\cap D$ induces an edge in $G$.
Otherwise, $\tau(C,D)$ is set to true.
One can easily see that if $\tau(C,D)$ is false and both cones $C$ and $D$
are used in the extension, then $H$ is not maximal $J_4$-free.
This test significantly prunes the search space.
Next, we assemble $q$-tuples of feasible cones such that each
pair of cones used passes the $\tau$-test. Each such $q$-tuple
defines one graph $H$. Finally, the isomorphic copies of graphs
are removed, and the remaining graphs $H$ are tested for $\chi(H)$.

The cases for small $q \le 2$ do not require the inner for-loop, since
they essentially reduce to the computation of feasible cones ($q=1$) and $\tau$ ($q=2$), respectively.
\hfill$\square$

\medskip
Note that larger values of input $\delta$ in {\bf A} produce fewer cones and thus allow for faster computation. Maximal $J_4$-free graphs $H$ with $\delta = 1$
are easy to characterize theoretically (in particular they have $\chi = 3$), and once they were confirmed by running {\bf A}, in all further computations we set $\delta \ge 2$. For most of our computations we use $\delta = 2$, but we also observe that when constructing graph families which are known to be $\chi$-vertex-critical, it is sufficient to set $\delta = \chi - 1$.

The values of the Ramsey numbers of the form $R(J_4,K_q)$
are known for all $q \le 6$ (cf. \cite{SRN}).
In particular, the values of importance to our computations are
$R(J_4,K_4)=11$, $R(J_4,K_5)=16$ and $R(J_4,K_6)=21$.
We use these values to determine the
value of parameter $q$ for algorithm ${\bf A}$
by applying an observation
that any $J_4$-free graph $H$ must have
$\alpha(H) \ge q$, provided $|V_H|=|V_G|+q \ge R(J_4,K_q)$.

Sections 3.2 and 3.3 list several sets of parameters for {\bf A} which
were used: $\mathcal{G}$ is taken from the cases reported in Table~\ref{t:smallcases},
or equal to $\mathcal{F}_v(2,3;J_4;14)$ or $\mathcal{F}_v(2^4;J_4;15)$,
and the ranges of other parameters are
$1 \le q \le 7$, $4 \le \chi \le 6$, and $2 \le \delta \le \chi-1$.

\subsection{Enumerations for Small Cases}

\begin{table}[h]
  \centering
  { \renewcommand{\arraystretch}{1.05} \small
    \begin{tabular}{|r||r|r|r|r|r|}
	\hline
    $n$ & all graphs & $J_4$-free & $\chi=2$ &
	$J_4$-free, $\chi=3$ & $J_4$-free, $\chi=4$\\
	\hline
      6 & 156 & 69 & 34 & 34 &0\\
      7 & 1044 & 255 & 87 & 167 &0\\
      8 & 12346 & 1301 & 302 & 998 &0\\
      9 & 274668 & 9297 & 1118 & 8175 &3\\
      10 & 12005168 & 97919 & 5478 & 92379 & 61\\
      11 & 1018997864 & 1519456 & 32302 & 1484866 & 2287\\
      12 & 165091172592 & 34270158 & 251134 & 33888537 & 130486\\
	  \hline
    \end{tabular}
	\caption{The number of nonisomorphic graphs $G$ by their type and the
number of vertices $n$, $6 \le n \le 12$. The corresponding sets of graphs were
obtained by using graph generator {\tt geng} of {\tt nauty} with
tests for $J_4$-free graphs and chromatic number $\chi(G)$.}
\label{t:smallcases}
    }
\end{table}

\begin{figure}[h!t]
\centering
\vspace{-1ex}
\includegraphics[scale=0.6]{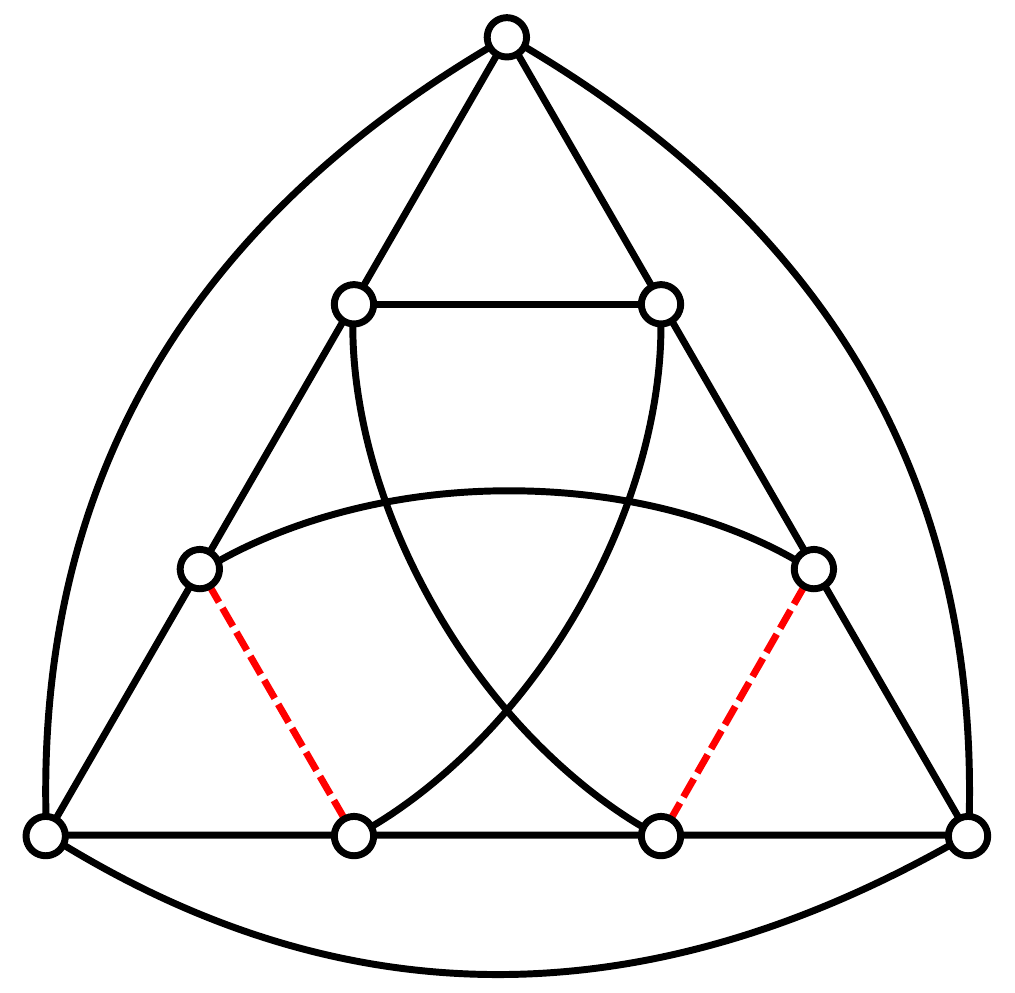}
\caption{$|\mathcal{F}_v(2,2,2;J_4;9)|=3.$
The 18-edge graph, which is {\tt H\{`Ypgj} in g6-format,
is formed by all depicted edges,
the other two graphs with 17 and 16 edges
are obtained by deleting the edges marked in red, in any order.}
\vspace{2ex}
\label{fig:X1}
\end{figure}

The results of our computations for small cases are summarized in Tables 2 and 3. The special graphs, which are witnesses for the exact values in Theorems 3, 4 and 6, are presented in Figures 1--4.

\subsubsection*{Notes to Table~\ref{t:smallcases}}
\begin{itemize}
\item
In each row $n$, the sum of entries for $\chi=2,3,4$ is one
less than the count of $J_4$-free graphs (because
the only missed graph has $\chi=1$). Note that
$\chi(G)=2$ implies that $G$ is $J_4$-free.
\item
The entries 0 and 3 in rows 8 and 9, respectively,
of the last column show that
$F_v(2,2,2;J_4)=9$ and $|\mathcal{F}_v(2,2,2;J_4;9)|=3$;
the three witnesses have 16, 17 and 18 edges, respectively,
and they are presented in Figure \ref{fig:X1}.
This part proves Theorem 3.
\item
Obtaining the next row of Table~\ref{t:smallcases} (for $n=13$)
by the same approach is doable
but only at an extraordinary computational cost. Thus, for $n \ge 13$
we first targeted only maximal $J_4$-free graphs. The results
are reported in Table~\ref{t:nonisomax}.
\end{itemize}

\begin{table}[h]
  \centering
  { \renewcommand{\arraystretch}{1.05} \small
    \begin{tabular}{|c||r|r|r|}
	\hline
    type of graphs & $n=13$ & $n=14$ & $n=15$ \\
	\hline
maximal $J_4$-free, $\chi=2$& 5 & 6 & 6 \\
maximal $J_4$-free, $\chi=3$& 15684 & & \\
maximal $J_4$-free, $\chi=4$& 4750 & 74738 & \\
maximal $J_4$-free, $\chi=5$& 0 & 0 & 1\\
	  \hline
    \end{tabular}
	\caption{Counts of nonisomorphic maximal $J_4$-free graphs $G$
	by their chromatic number $\chi=\chi(G)$
	and number of vertices $n$, for $13 \le n \le 15$. The results
	for $n \ge 14$ and $\chi \ge 4$ required
	significant computational resources.}
	\label{t:nonisomax}
    }
\end{table}

The graph families for $\chi \ge 3$ summarized in Table 3  were constructed using algorithm {\bf A}. The following lemma was used to determine the initial family of graphs $\mathcal{G}$. More details of 
how each entry with $\chi \geq 3$ was computed are listed in the notes to Table 3.

\begin{lemma}
    \label{lem:subgchi}
    Let $G$ be any graph with $\chi(G) \geq k$, and let $I \subseteq V(G)$ be any independent set in $G$. Then for
    $G'=G[V(G) \setminus I]$ we have $\chi(G') \geq k-1$.
    \end{lemma}
    \begin{proof}
Assume there exists an $I \subseteq V(G)$ such that the graph $G'$ induced in $G$ on $V(G) \setminus I$ has $\chi(G') \leq k-2$. If $V(G')$ is colored with 
$k-2$ colors, then all vertices in $I$ can be colored with the same $(k-1)$-st color, not used in the coloring of $V(G')$. This implies that $\chi(G) \leq k-1$, which is a contradiction.
\end{proof}

\subsubsection{Notes to Table~\ref{t:nonisomax}}

\begin{itemize}
\item
The row for $\chi=2$ shows the number of complete
bipartite graphs $K_{s,t}$ with $s+t=n$ and $s,t\ge 2$.

\item
\emph{Graphs with $n = 13, \chi \ge 3$.} Using the fact that $R(J_4,K_4) = 11$, we can see that any $J_4$-free
 graph $G$ of order at least 11 must have $\alpha(G) \geq 4$.
The graphs with $n = 13$ were obtained by {\bf A} in three different ways: via 3-, 2- and 1-vertex extensions of graphs with
 10, 11 and 12 vertices, respectively.
When $3 \leq \chi \leq 4$, we use $\delta = 2$.
When $\chi = 5$, we set $\delta = \chi - 1 = 4$, since the target graphs are known to be $\chi$-vertex-critical. This is because Lemma 9, $R(J_4,K_4) = 11$ and Table 2 imply that there is no $J_4$-free graph with $n = 12$ and $\chi = 5$.
\item
\emph{Graphs with $n = 14, \chi \ge 4$.} These graphs were obtained by computing 4-vertex extensions of
        the graphs with $n=10$ and $\chi\ge 3$ using $\delta = 2$. We set $\delta = 4$ when generating graphs with $\chi = 5$ since no graphs with $\chi = 5$ were found on 13 vertices.
\item
\emph{Graphs with $n = 15, \chi \ge 5$.} The unique maximal graph with $n=15$ and $\chi=5$
        was obtained by performing 3-vertex extensions of
        all graphs with $n=12$ and $\chi \geq 4$, and independently by 4-vertex extensions of all graphs with $n=11$ and $\chi \ge 4$. Since no graphs with $\chi = 5$ were found on 14 vertices, we
 set $\delta = 4$.

\item
Empty entries correspond to graphs whose full enumeration was
not attempted. These would be difficult to obtain, and they are not
relevant for this work. Still, many such graphs
were obtained as side result of other computations.
\item
The entry requiring most CPU time (about one CPU-week
if run on a single processor) was that for $\chi=4, n=14$.
It was obtained by applying algorithm {\bf A}
to make $J_4$-maximal 4-vertex extensions of
the 97918 graphs with $n=10$ and $\chi\ge 2$
(though using $\chi \ge 3$ would suffice).
Among the resulting 74738 graphs $G$, there are 24
of them for which $G \rightarrow (2,3)^v$. No graph
reported in column 13 satisfies this arrowing
(though, by Lemma 1, it would suffice to test
only 4750 graphs with $\chi=4$), and
thus $F_v(2,3;J_4)=14$.

\item
The complete set 
$\mathcal{F}_v(2,3;J_4;14)$ was obtained from the above 24
maximal $J_4$-free graphs by repeatedly deleting the edges
until they did not satisfy the arrowing.
This set consists of 212 nonisomorphic graphs,
with the number of edges ranging from 31 to 39,
the number of triangles from 8 to 10,
and the orders of their automorphism groups ranging from 1 to 8.
26 of these graphs are minimal. There exists a unique
(up to isomorphisms)
bicritical graph, namely the graph
$G_{14}$ shown in Figure \ref{fig:X2},
for which addition of any edge creates a $J_4$,
and deletion of any edge $e$ yields
$G_{14}-e \not\rightarrow (2,3)^v$.
This part proves Theorem~\ref{t:23j4}.

\begin{figure}[h!t]
\centering
\includegraphics[scale=0.75]{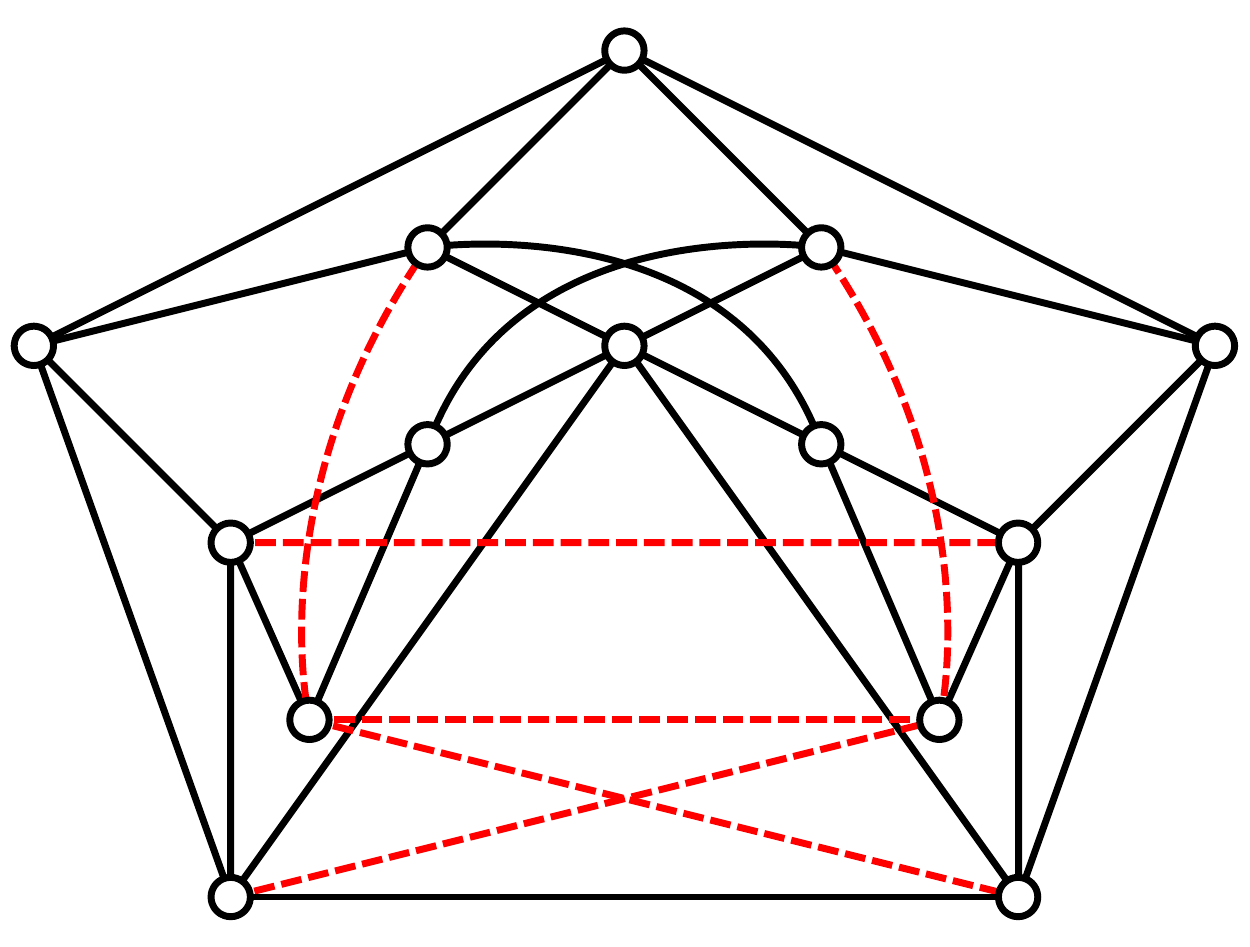}
\caption{Unique bicritical graph
$G_{14} \in \mathcal{F}_v(2,3;J_4;14)$.
Edges marked in red do not belong to any triangle in $G_{14}$.
The graph $G_{14}$, which is {\tt M?K\textunderscore iqg\textasciigrave QDqQXBpw?} in g6-format,
has 33 edges, 9 triangles and just one
non-trivial symmetry (left-right swap of the figure).}
\vspace{1ex}
\label{fig:X2}
\end{figure}

\begin{figure}
\centering
\includegraphics[scale=0.7]{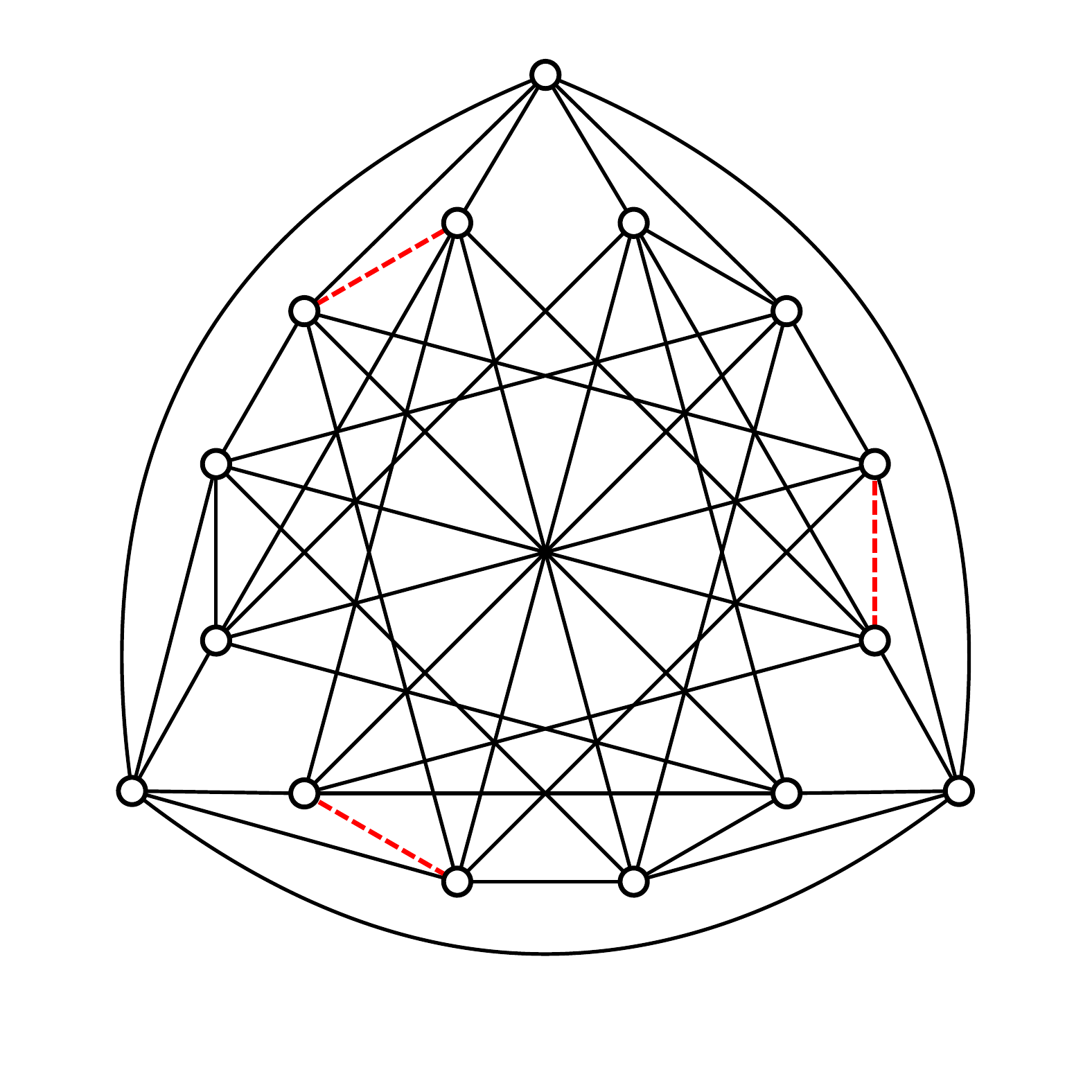}
\caption{A view of graphs in $\mathcal{F}_v(2^4;J_4;15)$.
The maximal graph on 45 edges,
which is {\tt N\{eCIhJSaWEfIuKqDeG} in g6-gormat,
is formed by all depicted edges. It is 6-regular.
Three vertices of the outer triangle form one orbit,
12 other vertices form the second orbit
(the center of the figure is not a vertex).
Three edges marked in red can be removed, in any order,
to give its subgraphs in $\mathcal{F}_v(2^4;J_4;15)$.
The 5-th graph is a subgraph of one with 44 edges.
The minimal graph on 42 edges (formed by the black edges)
has 72 automorphisms, more than the other four graphs.}
\label{fig:X3a}
\end{figure}

\begin{figure}[h!t]
\centering
\includegraphics[scale=0.7]{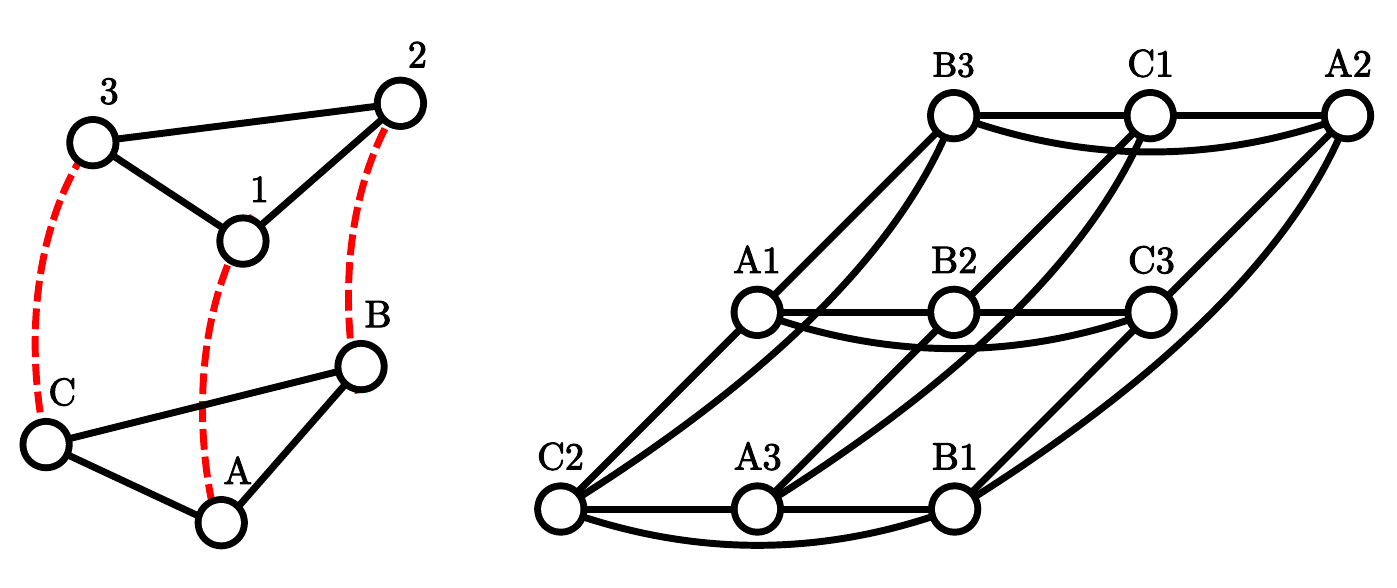}
\caption{Set view of $\mathcal{F}_v(2^4;J_4;15)$.
The 9-vertex $3 \times 3$ grid on the right has
6 triangles and 6 independent sets of 3 vertices. It is
a self-complementary graph.
The vertices \{A,B,C\} and \{1,2,3\} on the left are connected
to the grid by 18 edges as indicated by the labels.
Equivalently, 9 vertices of the grid on the right
connect to pairs of
vertices (one from each of two triangles) on the left.
The red edges can be dropped (in any order) yielding
graphs with 44, 43 and 42 edges, respectively. This figure
describes 4 graphs isomorphic to those in Figure 3, but
presenting them in a very different way.
The vertices of the middle row of the grid form
the triangle corresponding to the outer triangle in Figure 3.}
\vspace{2ex}
\label{fig:X3b}
\end{figure}

\item
The second most-expensive-to-obtain entry was that
for $n=15, \chi=5$. The entries in the last row of Table~\ref{t:nonisomax}
prove that $F_v(2^4;J_4)=15$.
The unique maximal graph in $\mathcal{F}_v(2^4;J_4;15)$
was obtained in two ways: as a 4-extension and 3-extension
of all graphs with $\chi=4$ on 11 and 12 vertices,
respectively.
\item
The complete set 
$\mathcal{F}_v(2^4;J_4;15)$ consists of 5 graphs
with 45, 44, 43, 43 and 42 edges, respectively.
Four of these graphs (except one on 43 edges) form
a chain presented in Figure \ref{fig:X3a}. The graph on
45 edges is formed by all edges, three red edges
can be removed (in any order) to give its subgraphs
which are also in $\mathcal{F}_v(2^4;J_4;15)$.
The fifth graph is a subgraph of one on 44 edges.
This part proves Theorem~\ref{t:2222j4}.
\item
Another view of the graphs in $\mathcal{F}_v(2^4;J_4;15)$
is presented in Figure~\ref{fig:X3b}.
The 9-vertex grid on the right has 6 independent sets
of 3 vertices. The vertices of triangles ABC and 123,
on the left, are connected to the grid as indicated
by the labels. The red edges can
be dropped (in any order), yielding graphs with
44, 43, 42 edges, also in $\mathcal{F}_v(2^4;J_4;15)$.
In this view we can easily see all 72 symmetries
of the minimal graph.
\end{itemize}

\subsection{Bounds for $F_v(2^5;J_4)$ and $F_v(2,2,3;J_4)$}

Easy bounds on the order of the smallest 6-chromatic
$J_4$-free graph, $F_v(2^5;J_4)$, are implied in prior
work by others on avoiding $K_3$ and $K_4$
(see Table~\ref{tab:k3k4} in Section~\ref{s:background}), namely:
$$
16=F_v(2^5;K_4) \le F_v(2^5;J_4) \le F_v(2^5;K_3) \le 40.
$$
We obtain much better bounds stated in Theorem~\ref{t:22222j4}:
$$
22 \le F_v(2^5;J_4) \le 25
$$
These bounds were computed as follows.
Take $S$ to be the Schl\"afli graph on 27 vertices \cite{Sei79}:
$S$ is a strongly regular graph of degree 16.
Its complement is $J_4$-free and
it has $\chi(\overline{S})=6$.
Removing from $\overline{S}$ any two adjacent
vertices with all incident edges yields
a 25-vertex witness to $F_v(2^5;J_4) \le 25$
(the g6-format of this graph is
{\tt XIPA@CQA\_KEBIIHKBHGicBXB\_w\}auURYbDu\_maULkdQTseOfpp?}).
Removing any further vertices reduces the chromatic number
below 6.
It is interesting to note that this witness graph is the 
unique 6-vertex-critical $(P_6, J_4, K_4)$-induced-free\footnote{There is no induced subgraph isomorphic to a $J_4$, $K_4$, or $P_6$.} graph on 25 vertices~\cite{JanP6}.

For the lower bound, since $R(J_4,K_6)=21$, any 21-vertex
$J_4$-free graph must have an independent set of 6 vertices.
Thus, any graph $G \in \mathcal{F}_v(2^5;J_4;21)$ can be obtained by
adding a 6-independent set (with 6 cones) to one of the 5
graphs in $\mathcal{F}_v(2^4;J_4;15)$. This was verified with
algorithm {\bf A} and no suitable graph was found with $\chi\ge 6$.
Thus $F_v(2^5;J_4)\ge 22$.
This part proves Theorem~\ref{t:22222j4}.

The bounds we have for $F_v(2,2,3;J_4)$ are rather weak, namely
$$F_v(2,3;J_4)+6=20 \le F_v(2,2,3;J_4) \le F_v(3,3;J_4) \le 135.$$
The upper bound follows from Theorem~\ref{t:33j4} and by an easy observation
that for any graph $G$, if
$G \rightarrow (3,3)^v$, then $G \rightarrow (2,2,3)^v$.
For the lower bound, suppose that $G \in \mathcal{F}_v(2,2,3;J_4;k)$.
Note that by Lemma~\ref{l:chromatic}, we must have $\chi(G)\ge 5$.
For $k=19$, since $R(J_4,K_5)=16$, we have $\alpha(G)\ge 5$,
and thus $G$ is a 5-vertex extension of at least one of the 212
graphs in $\mathcal{F}_v(2,3;J_4;14)$. Using again algorithm
{\bf A} we have found no suitable graph $G$, and hence $k>19$.
We attempted to use {\bf A} and other ad-hoc methods
to construct a witness $G$ for $k\ge 20$, but all such
searches failed. The complement of the Schl\"{a}fli graph
$\overline{S}$ does not arrow $(2,2,3)^v$. We also tested
8933 $J_4$-free graphs $G$ on 20 vertices with
$\chi(G)=\alpha(G)=5$ \cite{Rad} and found that
none of them arrows $(2,2,3)^v$.
For comparison with the cases for $K_4$-free graphs,
we note that it is not hard to check that $F_v(2,3;K_4)=7$
with the unique witness graph $K_7-C_7$ (cf.\@ Theorem 3 in \cite{LRU01}),
and it is known that $F_v(2,2,3;K_4)=14$~\cite{CR06}.

Finding any non-obvious bounds for
$F_v(2,3,3;J_4)$ or $F_v(3,3,3;J_4)$ is an interesting challenge which
we pose as a problem to work on.

\subsection{The $J_4$-free process and $F_v(3,3;J_4)$}

The triangle-free process begins with an empty graph of
order $n$, and iteratively adds edges chosen uniformly at
random, subject to the constraint that no triangle is formed.
The triangle-free process has been used to prove that
$$R(3,t) \ge \left({{\frac{1}{4}} + o(1)}\right){\frac{t^2}{\log t}},$$
which currently is
the best known lower bound for $R(3,t)$ obtained
by Bohman and Keevash in 2013/2019~\cite{BohmanKeevash2013a} and
independently by Fiz Pontiveros, Griffiths and Morris
in 2013/2020~\cite{FizGM2013a}.

Similarly to the triangle-free process, the $J_4$-free
process begins with an empty graph of order $n$, and
iteratively adds edges chosen uniformly at random, subject
to the constraint that no $J_4$ is formed.
The asymptotic properties of this process were
analyzed in~\cite{Pic14}.
We implemented the $J_4$-free process in C++ and
generated several graphs for which we then checked
the arrowing property.
The check was done by turning the arrowing property into
a Boolean formula and then using Boolean satisfiability
(SAT) solvers on the resulting formula.
The formula is computed as follows: for every triple
of vertices $(v_1,v_2,v_3)$, if they form a triangle,
we output the disjunctions
$$(\overline{v_1}\lor\overline{v_2}\lor\overline{v_3})\land(v_1\lor v_2\lor v_3)$$

A satisfying assignment for this subformula will
assign at least one of the vertices in $\{v_1,v_2,v_3\}$
to the value $\textsc{False}$ and at least one of them
to the value $\textsc{True}$. Taking $\textsc{False}$
and $\textsc{True}$ to be colors, it is clear
that for a graph $G$ the formula
$$\bigwedge\limits_{\substack{(v_1,v_2,v_3)\in V_G\\\textnormal{\scriptsize s.t. }
G[\{v_1,v_2,v_3\}]\sim K_3}}(\overline{v_1}\lor\overline{v_2}\lor\overline{v_3})
\land(v_1\lor v_2\lor v_3)$$

\noindent
is satisfiable if and only if there is a way to assign colors
to the vertices of $G$ that avoids monochromatic triangles.
We are thus searching for $J_4$-free graphs $G$ that yield
unsatisfiable instances, as these witness the bound
$F_v(3,3;J_4)\leq|V_G|$. The smallest such graph
we were able to find has 135 vertices, thus establishing that
$$F_v(3,3;J_4) \le 135.$$
This part proves Theorem~\ref{t:33j4}.

It is easy to see that $G \rightarrow (3,3)^v$
implies $K_1+G \rightarrow (3,3)^e$, where the graph
$K_1+G$ is obtained from $G$ by adding one new vertex
connected to all of $V_G$. By applying this implication we
can also see that $F_e(3,3;K_1+H) \le F_v(3,3;H)+1$.
The latter inequality is tight for $H=K_4$, as it was shown
that $F_e(3,3;K_5)=15$ and $F_v(3,3;K_4)=14$~(upper bounds in \cite{Nen81}, lower bounds in  \cite{PRU99}).
Now, using similar steps for $H=J_4$, by Theorem~\ref{t:33j4}
we obtain $F_e(3,3;J_5) \leq 136$.
We observe that by the monotonicity with respect
to the avoided graph $H$ we have
$F_e(3,3;K_5) \le F_e(3,3;J_5) \le F_e(3,3;K_4)$.

\medskip
The software development for this project, pilot runs and experiments spanned several months. Now, the final run of everything is doable in about a day in a standard lab with dozen of machines, each with 16 cores.


\section{Remarks and Some Open Problems}

Obtaining good bounds for concrete vertex Folkman numbers is
often difficult. The results presented in this paper are special
for the parameters involving $J_4$ but we hope that they may extend
to other or more general vertex cases. The techniques used here might
be not easily transferable to obtain new bounds on edge Folkman numbers,
not the least because just testing edge-arrowing is typically much more
difficult than testing vertex-arrowing.

\bigskip
We close this paper by posing some related open problems.

\begin{problem}\label{lower}
Give a general lower bound for $F_v(3^r;J_4)$,
or any nonobvious lower bound for $F_v(3,3;J_4)$,
which are not easily implied by known bounds for
other more studied parameters.
\end{problem}

Similarly, we do not know much about the cases
like $F_v(2^r;K_4)$, $F_v(3^r;K_4)$, or $F_v(4^r;J_5)$:
no general methods are known to obtain good lower
or upper bounds.

\begin{problem}
Does there exist a $J_4$-free graph $G$ such that
every set of $|V_G|/2$ vertices induces a triangle? 
\end{problem}

\noindent
If not, this could help in the analysis of $F_v(3,3;J_4)$.
We may consider this problem in a more general case. 
By using density arguments, it is known how to obtain
upper bounds on $F_v(3^r;K_4)$, but not
on $F_v(3,3;J_4)$ or $F_v(3^r;J_4)$.

%
%
%

Except for $F_e(3,4;K_5) \ge 22$ \cite{XS10},
we do not know of any other nonobvious bounds for:
(a) $F_v(3,4;J_5)$, (b) $F_v(4,4;J_5)$, 
(c) $F_e(3,4;K_5)$, or (d) $F_e(4,4;J_5)$.
Case (a) might be solvable with computational methods.
We also think that case (b), while far from easy, is easier
than (c) and much easier than (d). For a
similar case of $F_v(4,4;K_5)$,
the best known bounds were obtained in \cite{Bik18,Fv445}:
$$19 \le F_v(2,2,2,4;K_5) \le F_v(4,4;K_5) \le 23.$$

Finally, we state a related existence problem, which
was already raised in an earlier work \cite{XLR18},
and which is also described at the end of Section 2.

\begin{problem}
{\rm (a)}
$\mathcal{F}_e(3,3;K_1+C_4)=\emptyset?$
\hspace{.3cm}
{\rm (b)}
$\mathcal{F}_e(3,3;\overline{P_2 \cup P_3})=\emptyset?$
\end{problem}
We note that the YES answer to part (a) implies YES answer
to part (b), which eliminates one YES/NO combination
of answers (out
of four possible ones). This problem can be rephrased in
some interesting ways. For example,
part (a) is equivalent to the following question:
{\em Does there exist any $W_5$-free graph
which is not a union of two triangle-free graphs?}

\section*{Acknowledgement}

This research was partially supported by
the National Natural Science Foundation of China (11361008).
David E. Narváez was supported by NSF grant CCF-2030859 to the
Computing Research Association for the CIFellows Project.



\end{document}